\documentclass[12pt]{amsart}
\usepackage{array}
\usepackage{amstext}
\usepackage{amsfonts}
\usepackage{amsmath}
\usepackage{amssymb}
\usepackage{mathtools}
\usepackage{color}
\usepackage{enumerate}
\usepackage{filecontents}
\usepackage[colorlinks=true, linkcolor=blue]{hyperref}

\usepackage{cancel}
\usepackage{changes}

\usepackage{tikz}
\usepackage{tikz-cd}
\usetikzlibrary{arrows, positioning}
\usepackage[all]{xy}
\tikzset{Mylong/.style={text width=3.1cm, align=center}, myarr/.style={->, double equal sign distance, -implies}}

\newtheorem{theorem}{Theorem}[section]

\newtheorem{lemma}[theorem]{Lemma}
\newtheorem{proposition}[theorem]{Proposition}

\theoremstyle{definition}

\newtheorem{example}[theorem]{Example}
\theoremstyle{remark}
\newtheorem{remark}[theorem]{Remark}

\title{Conjugacy classes of positive $3$-braids}
\author{Kui-Yo Chen}
\address{Department of Mathematics, National Cheng Kung University, Tainan 701401}
\email{kuiyochen1230@gmail.com}

\author{Yat-Hin Suen}
\address{Department of Mathematics, National Cheng Kung University, Tainan 701401}
\email{yhsuen@gs.ncku.edu.tw}

\date{\today}

\begin{document}
\begin{abstract}
The conjugacy problem in braid groups has been extensively studied, particularly from an algorithmic perspective. Established methods based on Garside structures, such as initial summit sets and super summit sets, provide effective procedures for determining whether two braids are conjugate.

In contrast, explicit structural descriptions of conjugacy classes are less frequently addressed. Although cyclic sliding offers a powerful mechanism for navigating distinguished subsets within a conjugacy class, it is well known that conjugate braids cannot, in general, be obtained from one another solely through iterated cyclic sliding.

In this paper, we provide a direct and explicit characterization of the conjugacy classes of positive $3$-braids. Specifically, for any given positive $3$-braid, we determine all of its conjugates in a concrete and closed form.
\end{abstract}

\maketitle

\section{Introduction}

The conjugacy problem in braid groups is a central algorithmic challenge in combinatorial group theory.
Since Garside's foundational work, braid groups have served as a testing ground for developing algebraic and geometric methods to solve conjugacy problems in non-abelian groups.

Let $B_n$ denote the braid group on $n$ strands.
We denote the $n-1$ standard generators of $B_n$ by $\sigma_i$ for $i=1,\dots,n-1$.
The relations between these generators are
\begin{equation}\label{braid_relation1}
\sigma_i\sigma_j\sigma_i=\sigma_j\sigma_i\sigma_j \quad \text{if } |i-j|=1,
\end{equation}
and
\begin{equation}\label{braid_relation2}
\sigma_i\sigma_j=\sigma_j\sigma_i \quad \text{if } |i-j|\ge 2.
\end{equation}



The positive braid monoid $B_n^+$, generated by the same elements and relations, embeds naturally into $B_n$ \cite{The_braid_group_and_other_groups}. Geometrically, conjugacy classes in $B_n$ correspond to isotopy classes of closed braids in the solid torus \cite{Braid_groups}.

Extensive research has focused on effective solutions to the conjugacy problem.
Techniques based on Garside structures led to the development of summit sets, super summit sets, and ultra summit sets \cite{Algorithms_for_positive_braids, The_braid_group_and_other_groups, positive_conjugate}.
Subsequent methods, such as cyclic sliding, further improved algorithmic efficiency and clarified the internal dynamics of conjugacy classes \cite{A_new_approach, Word_processing, Basic_results, cyclic_sliding}.
See also \cite{Quasipositivity}.

Despite these advances, most existing literature emphasizes the decision of conjugacy rather than the explicit structural description of conjugacy classes.
Even for braids with few strands, a direct characterization of all conjugates is not always available.

This paper shifts the focus from decision procedures to structural classification in the case of positive $3$-braids.
While cyclic sliding provides a systematic way to move within certain distinguished subsets of a conjugacy class, it is well understood that conjugate braids are not always connected through iterated cyclic sliding alone.


We introduce the notion of \emph{cyclic equivalence} for positive braids,
which arises from cyclic permutations of braid words.
We show that conjugacy in $B_3^+$ can be characterized by cyclic equivalence and the reflection of braids.
Our main result is as follows:

\begin{theorem}[Main Theorem \ref{main}]
Let $a,b\in B_3^+$.
Then $a$ and $b$ are conjugate if and only if either $a$ and $b$ are cyclic equivalent, or $a$ is cyclic equivalent to the reflection of $b$.

Moreover, if $a$ has positive infimum, then $a$ and $b$ are conjugate if and only if they are cyclic equivalent.
\end{theorem}
\begin{remark}
The theorem shows that the conjugacy class of a positive $3$-braid consists of at most two cyclic equivalence classes.
\end{remark}


In addition to the main classification, we provide a proposition (Proposition \ref{prop}) that characterizes the necessary and sufficient conditions under which a conjugacy class in $B_3^+$ coincides with a single cyclic equivalence class.
This result further refines the structural relationship between conjugacy and cyclic permutations in the $3$-strand case.


\section{Preliminaries}

Let $\Delta_n\in B_n$ be the Garside element defined by
\[
\Delta_n := \Pi_{n-1}\Pi_{n-2}\cdots \Pi_1 \in B_n,
\]
where $\Pi_k := \sigma_1\sigma_2\cdots\sigma_k$.

\medskip

We say that $a,b\in B_n^+$ are {\it cyclic equivalent}, denoted by
$a\sim_\circlearrowright b$, if there exists a finite sequence of positive
braids
\[
a=a_0,a_1,\dots,a_m=b
\]
such that for each $i=0,1,\dots,m-1$, the braid $a_i$ admits a positive
word representative $w_1w_2$ for which $w_2w_1$ represents the braid
$a_{i+1}$.  
We call the transition from $a_i$ to $a_{i+1}$ a {\it cyclic word move}.

\medskip

Recall that every braid $a$ admits a standard normal form
\[
a=\Delta_n^m P,
\]
where $m\in\mathbb Z$ and $P$ is a positive braid not containing
$\Delta_n$ as a factor.  
The integer $m$ is called the {\it infimum} of $a$, and is denoted by
$\inf(a)=m$.

\medskip

We define the {\it reflection} of $a$ by
\[
\overline{a}:=\Delta_n^{-1}a\Delta_n=\Delta_n a\Delta_n^{-1},
\qquad a\in B_n,
\]
so that $\overline{\overline{a}}=a$.  
Note that $\overline{\sigma_i}=\sigma_{n-i}$ and
$\overline{ab}=\overline{a}\,\overline{b}$.  
In particular, if $a\in B_n^+$, then $\overline{a}\in B_n^+$.

\begin{lemma}\label{zero_infimum}
All positive $3$-braids with zero infimum are precisely of the forms
\begin{itemize}
\item $\sigma_i^{a_1}\sigma_j^{a_2}\sigma_i^{a_3}\sigma_j^{a_4}\cdots\sigma_i^{a_k}$,
\item $\sigma_i^{a_1}\sigma_j^{a_2}\sigma_i^{a_3}\sigma_j^{a_4}\cdots
\sigma_i^{a_{k-1}}\sigma_j^{a_k}$,
\end{itemize}
where $k\ge 0$, $\{i,j\}=\{1,2\}$, $a_1,a_k\ge1$, and
$a_\ell\ge2$ for $\ell=2,\dots,k-1$.
\end{lemma}

\begin{remark}
For any $3$-braid $a$ written in the standard normal form
$\Delta_3^m P$, the factor $P$ is of one of the forms described in
Lemma~\ref{zero_infimum}. Moreover, $m\ge 0$ if and only if
$a\in B_3^+$.
\end{remark}

\begin{remark}\label{zero_infimum_remark}
Any braid $a$ with $\inf(a)=0$ has a unique word representative.

Moreover, if $a$ is either of the first form in
Lemma~\ref{zero_infimum} or of the second form with
$a_1,a_k\ge2$, then $a\sim_\circlearrowright b$ implies that there
exist positive words $w_1,w_2$ such that
\[
a=w_1w_2, \qquad b=w_2w_1 .
\]
That is, a single cyclic word move suffices to obtain $b$ from $a$.
\end{remark}
\begin{proof}
We show that any braid $a$ with $\inf(a)=0$ has a unique word representative.

Let $a$ and $b$ be two word representative of some positive $3$-braid.
Since the only relation in $B_3^+$ is $\sigma_1\sigma_2\sigma_1=\sigma_2\sigma_1\sigma_2$,
there exist a sequence of words $a=w_0, w_1, \cdots, w_m=b$ such that
for any $k=0, 1, \cdots, m-1$, $w_k=uvw$ and $w_{k+1}=uv'w$ for some positive words $u,w$ and
${v,v'}=\{\sigma_1\sigma_2\sigma_1,\sigma_2\sigma_1\sigma_2\}$.

Notice that whenever $m>0$, $a=w_0$ contains a Garside element $v=\Delta_3$.
Thus, $m>0$ implies $inf(a)>0$.
Therefore, any braid $a$ with $\inf(a)=0$ has a unique word representative.
\end{proof}

Throughout this paper, we denote $[a] := \text{cl}(a) \cap B_3^+$, where $\text{cl}(a)$ is the conjugacy class of $a$ in $B_3$. For brevity, we refer to $[a]$ simply as the conjugacy class of $a$.

\section{Conjugacy of positive $3$-braids}
In this section we prove the main theorem. 
We begin by analyzing conjugation by simple elements on positive $3$-braids. 
By \cite[Proposition 4.10]{positive_conjugate}, every conjugation can be decomposed into conjugations by simple elements, so it suffices to understand these elementary cases. 
We illustrate the result with an example and finally characterize those positive $3$-braids whose conjugacy class coincides with their cyclic equivalence class.

\begin{lemma}\label{one_step}
Let
\[
S:=\{1,\sigma_1,\sigma_2,\sigma_1\sigma_2,\sigma_2\sigma_1,\Delta_3\}
\]
be the set of simple $3$-braids. If $a\in B_3^+$, then for each
\[
s\in S_a^+:=\{s\in S \mid \inf(s^{-1}as)\ge 0\},
\]
the braid $s^{-1}as$ is cyclic equivalent either to $a$ or to $\overline{a}$.
\end{lemma}

\begin{proof}
Write $a$ in standard normal form as
$a=\Delta_3^m P$,
where $m=\inf(a)\ge 0$ and $P$ is of one of the forms described in Lemma~\ref{zero_infimum}.
We divide the possibilities for $P$ into the following cases:
\begin{enumerate}
    \item\label{P1} $P=1$;
    \item\label{P2} $P=\sigma_i P' \sigma_i^2$ or $P=\sigma_i$;
    \item\label{P3} $P=\sigma_i P' \sigma_j\sigma_i$;
    \item\label{P4} $P=\sigma_i P' \sigma_j$
\end{enumerate}
for some $P'\in B_3^+$ and $\{i,j\}=\{1,2\}$.

Assume first that $m\ge 1$. For any $s\in S_a^+$, we have
\[
s^{-1}as=s^{-1}\Delta_3\Delta_3^{m-1}Ps.
\]
Set $s':=s^{-1}\Delta_3\in B_3^+$. Then
\[
s^{-1}as=s'\Delta_3^{m-1}Ps \sim_\circlearrowright ss'\Delta_3^{m-1}P
=\Delta_3^mP=a.
\]

It remains to consider the case $m=0$, so that $a=P$.

In case~\ref{P1}, the statement is trivial.
If $s=1$ or $s=\Delta_3$, then $s^{-1}as$ is equal to either $a$ or $\overline{a}$.

Suppose next that $s=\sigma_i$.
In cases~\ref{P2}, \ref{P3}, and~\ref{P4}, the braid $\sigma_i^{-1}a\sigma_i$ is cyclic equivalent to $a$.

Now suppose that $s=\sigma_j$. In cases~\ref{P2} and~\ref{P4}, one has
$\inf(s^{-1}as)=-1$,
hence $\sigma_j\notin S_a^+$. In case~\ref{P3},
\[
s^{-1}as
=(\Delta_3^{-1}\sigma_j\sigma_i)\sigma_i P' \sigma_j\sigma_i\sigma_j
=\Delta_3^{-1}\sigma_j\sigma_i^2 P' \Delta_3
=\overline{\sigma_j\sigma_i^2 P'}
\sim_\circlearrowright \overline{\sigma_i P' \sigma_j\sigma_i}
=\overline{a}.
\]

Suppose now that $s=\sigma_i\sigma_j$. In cases~\ref{P2}, \ref{P3}, and~\ref{P4}, we may write $P=\sigma_iP''$ for some positive braid $P''$. Then
\[
s^{-1}as=\sigma_j^{-1}(P''\sigma_i)\sigma_j,
\]
which reduces to the previous case.

Finally, suppose that $s=\sigma_j\sigma_i$. In case~\ref{P4},
\[
s^{-1}as=\Delta_3^{-1}\sigma_i(\sigma_iP'\sigma_j)\sigma_j\sigma_i
\]
has infimum $-1$, hence $s\notin S_a^+$. In cases~\ref{P2} and~\ref{P3},
\[
s^{-1}as
=s^{-1}(\sigma_iP'\sigma_i)s
=\Delta_3^{-1}\sigma_i(\sigma_iP'\sigma_i)\sigma_j\sigma_i
=\Delta_3^{-1}\sigma_i^2P'\Delta_3
=\overline{\sigma_i^2P'}
\sim_\circlearrowright \overline{\sigma_iP'\sigma_i}
=\overline{a}.
\]
This completes the proof.
\end{proof}

The previous lemma allows us to reduce conjugacy to cyclic equivalence and reflection.

\begin{theorem}\label{main}
Let $a,b\in B_3^+$. Then $a$ and $b$ are conjugate if and only if
$a\sim_\circlearrowright b$ or $a\sim_\circlearrowright \overline{b}$.
In other words,
\[
[a]=[a]_\circlearrowright\cup[\overline{a}]_\circlearrowright.
\]
Moreover, if $\inf(a)\ge 1$, then $a$ and $b$ are conjugate if and only if they are cyclic equivalent.
Equivalently,
\[
[a]=[a]_\circlearrowright,
\]
if $\inf(a)\ge 1$.
\end{theorem}
\begin{proof}
By \cite[Proposition 4.10]{positive_conjugate}, Lemma~\ref{one_step} implies that $a$ and $b$ are conjugate if and only if
$a\sim_\circlearrowright b$ or $a\sim_\circlearrowright \overline{b}$.

It remains to prove the final statement. Write
$a=\Delta_3^mP$
in standard normal form, where $m\ge 1$ and $P$ is a positive word. Then
\[
a=\Delta_3\Delta_3^{m-1}P
=\overline{\Delta_3^{m-1}P} \Delta_3
\sim_\circlearrowright
\Delta_3 \overline{\Delta_3^{m-1}P}
=\overline{\Delta_3^mP}
=\overline{a}.
\]
Hence, if $b\sim_\circlearrowright \overline{a}$, then also $b\sim_\circlearrowright a$. Therefore, when $\inf(a)\ge 1$, the braids $a$ and $b$ are conjugate if and only if $a\sim_\circlearrowright b$.
\end{proof}

\begin{remark}
The statement of Theorem~\ref{main} does not extend to $n>3$ in general. For example, when $n=4$, the braids $\sigma_1$ and $\sigma_2$ have the same closed braid type in the solid torus, but they are neither cyclic equivalent nor cyclic equivalent to each other's reflections.
\end{remark}

\begin{example}\label{length4_example}
\begin{align*}
[\sigma_1^4]
&=[\sigma_1^4]_\circlearrowright\cup [\sigma_2^4]_\circlearrowright=\{\sigma_1^4,\sigma_2^4\};\\
[\sigma_1^2\sigma_2^2]
&=[\sigma_1^2\sigma_2^2]_\circlearrowright\cup [\sigma_2^2\sigma_1^2]_\circlearrowright\\
&=[\sigma_1^2\sigma_2^2]_\circlearrowright\\
&=\{\sigma_1^2\sigma_2^2,\sigma_1\sigma_2^2\sigma_1,\sigma_2^2\sigma_1^2,\sigma_2\sigma_1^2\sigma_2\};\\
[\sigma_1^3\sigma_2]
&=[\sigma_1^3\sigma_2]_\circlearrowright\cup [\sigma_2^3\sigma_1]_\circlearrowright
\quad(\text{use the fact }\sigma_1^2\sigma_2\sigma_1=(\sigma_1\sigma_2)^2=\sigma_2\sigma_1\sigma_2^2)\\
&=\{\sigma_1^3\sigma_2,(\sigma_1\sigma_2)^2,(\sigma_2\sigma_1)^2,\sigma_2\sigma_1^3,\sigma_2^3\sigma_1,\sigma_1\sigma_2^3\}\\
&=[(\sigma_1\sigma_2)^2]_\circlearrowright
\quad(\text{since }\inf((\sigma_1\sigma_2)^2)=1)\\
&=[\sigma_1^3\sigma_2]_\circlearrowright.
\end{align*}

These are all conjugacy classes of positive $3$-braids with length $4$.
\end{example}

We now turn to the question of when the conjugacy class of a positive $3$-braid coincides with its cyclic equivalence class.

Example~\ref{length4_example} shows that the conjugacy class of a positive $3$-braid may coincide with its cyclic equivalence class even when the infimum is zero, such as $\sigma_1^2\sigma_2^2$ and $\sigma_1^3\sigma_2$.
The following lemma and proposition give a necessary and sufficient condition for this phenomenon.

We say that the expression $(c\overline{c})^{\ell_1}$ for some positive word $c$ is {\bf minimal}
if $c \overline{c}\neq (d \overline{d})^{\ell_2}$ for any positive word $d$ and $\ell_2\geq 2$.
Note that a minimal expression exists for any word of the form $(c\overline{c})^\ell$.

\begin{lemma}
Let $a$ be a positive braid in $B_3$ with $\inf(a)=0$.
Suppose 
$a$ is either the first form in Lemma \ref{zero_infimum} or
the second form in Lemma \ref{zero_infimum} with $a_1, a_k\geq 2$
and satisfies
$a\sim_\circlearrowright \overline{a}$,
then $a=(c \overline{c})^\ell$ for some positive word $c$ and $\ell\in \mathbb{N}$
and therefore $a$ has a minimal expression.

Furthermore,
if $a=w_1w_2$, $\overline{a}=w_2w_1$ and $a=(c \overline{c})^\ell$ is a minimal expression,
then $w_1$ must be $(c \overline{c})^{\ell_3}c$ for some $\ell_3\geq 0$.
\end{lemma}
\begin{proof}
We argue by induction on the length $l(a)$ of the braid.
For braids of length at most $18$, the statement can be verified by computer.
(The number $18$ is just a good enough number for programming.)

Now, suppose $l(a)\geq \ell > 18$ and assuming this Lemma is true for all length less then $\ell$.

Let $a=w_2'w_1$ such that $\overline{a}=w_1w_2'$.
Without loss of generality, assuming $l(w_1)\leq l(w_2')$.
Then $w_2'w_1=\overline{\overline{a}}=\overline{w_1}\overline{w_2'}$.
So we may write $w_2'=\overline{w_1}w_2$ for some $w_2$ with $l(w_2)\geq 0$.
Then we have
\begin{equation}\label{word_induction_eq}
w_2w_1=w_1\overline{w_2}.
\end{equation}

We distinguish four cases according to the lengths of $w_1$ and $w_2$,
$l(w_2)=0$, $0<l(w_1)<l(w_2)$, $l(w_1)>l(w_2)>0$ and $l(w_1)=l(w_2)$.

For the case $l(w_2)=0$, $a=\overline{w_1}w_1 = c\overline{c}$ by taking $c=\overline{w_1}$.

For the case $0<l(w_1)<l(w_2)$,
comparing the prefixes of length $l(w_2)$ on both sides of equation (\ref{word_induction_eq}),
we may write $w_2 = w_1w_3$ with $w_3$ being a prefix of $\overline{w_2}$.
Then we have $w_3w_1=\overline{w_1}\overline{w_3}$.
This means that $w_3w_1=\overline{w_1}\overline{w_3}\sim_\circlearrowright \overline{w_3}\overline{w_1}=\overline{w_3w_1}$.
By induction hypothesis, $w_3w_1$ has a minimal expression $(c\overline{c})^{\ell_1}$ for $\ell_1\geq 1$.
Moreover, also by induction hypothesis, $w_3 = (c\overline{c})^{\ell_2}c$ and $w_1 = \overline{c}(c\overline{c})^{\ell_1-\ell_2-1}$ for some $\ell_2\geq 0$.
Thus, $w_2=\overline{c}(c\overline{c})^{\ell_1-1}c$ and $a=\overline{w_1}w_2w_1=(c\overline{c})^{3\ell_1-2\ell_2-1}$,
where $3\ell_1-2\ell_2-1\geq 2$.

For the case $l(w_1)>l(w_2)>0$, write $w_1 = w_2w_{1,1}$.
So $w_2w_{1,1}=w_{1,1}\overline{w_2}$.
If $l(w_{1,1})>l(w_2)$, write $w_1 = w_2w_{1,1} = w_2^2w_{1,2}$.
Keep doing this until, $l(w_{1,\ell_3})<l(w_2)$,
so that $w_1 = w_2^{\ell_3}w_{1,\ell_3}$ and $w_2w_{1,\ell_3}=w_{1,\ell_3}\overline{w_2}$
for $\ell_3\geq 1$.
Then it is similar to the previous case. Applying induction hypothesis again, we have
$w_2=\overline{c}(c\overline{c})^{\ell_1-1}c$ and $w_{1,\ell_3}=\overline{c}(c\overline{c})^{\ell_1-\ell_2-1}$.
So $w_1=w_2^{\ell_3} w_{1,\ell_3}
=\overline{c}(c\overline{c})^{\ell_1\ell_3-1}c\overline{c}(c\overline{c})^{\ell_1-\ell_2-1}
=\overline{c}(c\overline{c})^{\ell_1\ell_3-1 + \ell_1-\ell_2}
$.
Hence, $a=\overline{w_1}w_2w_1=(c\overline{c})^{2\ell_1\ell_3 + 3\ell_1-2\ell_2-1}$ for $2\ell_1\ell_3 + 3\ell_1-2\ell_2-1\geq 4$.

The case $l(w_1)=l(w_2)$ is impossible, since it would imply $w_1=\overline{w_1}$, which cannot occur in $B_3^+$.

To complete the induction step,
it remain to show the second statement of the lemma.
Let $a=(c \overline{c})^{\ell_4}$ be a minimal expression.
Write $c=w_3w_4$, so that either
\[
a=(c \overline{c})^{\ell_5}w_3w_4 \overline{c}(c \overline{c})^{\ell_4-\ell_5-1}
\text{ and }
\overline{a}=w_4 \overline{c}(c \overline{c})^{\ell_4-\ell_5-1}(c \overline{c})^{\ell_5}w_3\]
or
\[
a=(c \overline{c})^{\ell_5}c \overline{w_3w_4}(c \overline{c})^{\ell_4-\ell_5-1}
\text{ and }
\overline{a}=\overline{w_4}(c \overline{c})^{\ell_4-\ell_5-1}(c \overline{c})^{\ell_5}c\overline{w_3}.\]
We need to show either $l(w_4)=0$ in first case or $l(w_3)=0$ in second case.
The second case is similar to the first, so we will only prove the first case.

In the first case,
$
(c \overline{c})^{\ell_5}w_3w_4 \overline{c}(c \overline{c})^{\ell_4-\ell_5-1}
=
\overline{\overline{a}}
=
\overline{w_4} c(\overline{c}c)^{\ell_4-1}\overline{w_3}
$.
By comparing the prefixes of length $l(c)$ on both sides implies $w_3w_4=c=\overline{w_4}w_3$.
This is analogous to the Equation \ref{word_induction_eq} with word order reversing.
With similar discussion, $4$ cases above fail except the first one $l(w_4)=0$, otherwise $c \overline{c}$ can be written in to $(d\overline{d})^{\ell_6}$ for some $d$ and $\ell_6\geq 2$. Thus, induction step is proved.
\end{proof}

\begin{proposition}\label{prop}
Suppose $a$ and $\overline{a}$ are positive braids that are conjugate in $B_3$. 
Then $a$ and $\overline{a}$ are cyclic equivalent if and only if at least one of the following conditions fails:
\begin{enumerate}
\item\label{propP1:label}
$\inf(a)=0$;
\item\label{propP2:label}
$a$ is either of the first form in Lemma \ref{zero_infimum} or
 the second form in Lemma \ref{zero_infimum} with $a_1, a_k\geq 2$;
\item\label{propP3:label}
$a\neq(c \overline{c})^\ell$ for any positive word $c$ and $\ell\in \mathbb{N}$.
\end{enumerate}
\end{proposition}
\begin{proof}
Suppose that $a\sim_\circlearrowright \overline{a}$ and $a$ satisfies (\ref{propP1:label}) and (\ref{propP2:label}).
Then the previous lemma implies $a=(c \overline{c})^\ell$ for some positive word $c$ and $\ell\in \mathbb{N}$.
That is, (\ref{propP3:label}) fails.

Conversely,
if (\ref{propP1:label}) fails, Theorem~\ref{main} shows that $a\sim_\circlearrowright \overline{a}$;
If (\ref{propP2:label}) fails, $a=\sigma_i\sigma_j P' \sigma_j$ or $\sigma_i P' \sigma_i\sigma_j$ for some $P'\in B_3^+$ by assumption,
then
\[a\sim_\circlearrowright \Delta_3 P',\]
since $\Delta_3=\sigma_i\sigma_j\sigma_i=\sigma_j\sigma_i\sigma_j$.
Thus, Theorem~\ref{main} shows that
\[a\sim_\circlearrowright \Delta_3 P' \sim_\circlearrowright \overline{\Delta_3 P'}\sim_\circlearrowright \overline{a};\]
Finally, if (\ref{propP3:label}) fails, then
\[a=(c \overline{c})^\ell \sim_\circlearrowright \overline{c}(c\overline{c})^{\ell-1}c=(\overline{c} c)^\ell = \overline{a}.\]
This completes the proof.
\end{proof}

\noindent
\phantomsection\addcontentsline{toc}{section}{Acknowledgements}
\textbf{Acknowledgements.}
The authors would like to thank Hongtaek Jung for a careful reading of the manuscript and for many valuable suggestions that helped improve the clarity of the presentation.
We also thank
Chun-Wei Lee
for their assistance and their useful comments.
The work of Y.-H. Suen was supported by the Ministry of Education Republic of China (Taiwan) - Yushan Fellow Program H114-A116B.

\vspace{.1in}
\end{document}